\documentclass[reqno,a4paper,12pt]{amsart}

\numberwithin{equation}{section}
\usepackage[utf8]{inputenc}
\usepackage{amsmath, amsthm, amssymb, amscd, accents, bm, amsfonts}
\usepackage{mathtools}
\usepackage{url}
\usepackage{mathrsfs,dsfont}
\usepackage{mathtools}
\usepackage{mdwlist}
\usepackage{paralist}
\usepackage{datetime}
\usepackage{hyperref}
\usepackage{colonequals}
\usepackage{xcolor}
\usepackage[sort,nocompress]{cite}

\mathtoolsset{showonlyrefs}

\newtheorem{definition}{Definition}[section]
\newtheorem{theorem}[definition]{Theorem}
\newtheorem{lemma}[definition]{Lemma}
\newtheorem{corollary}[definition]{Corollary}
\newtheorem{proposition}[definition]{Proposition}

\def\R{{\mathbb R}}

\allowdisplaybreaks

\begin{document}

\title[Cheeger's isoperimetric problem for Gaussian mixtures]{Cheeger's isoperimetric problem for \\ Gaussian mixtures}

\author[Lukas Liehr]{Lukas Liehr}
\address{Department of Mathematics, Bar-Ilan University, Ramat-Gan 5290002, Israel}
\email{lukas.liehr@biu.ac.il}

\date{\today}

\begin{abstract}
In any dimension $n$, we determine the Cheeger constant and the Cheeger sets of the Gaussian mixture $$\mu(x) = p\gamma(x-a) + (1-p)\gamma(x-b),$$ where $p \in [0,1]$, $a,b \in \R^n$, and $\gamma : \mathbb{R}^n \to (0,\infty)$ denotes a Gaussian. In particular, we characterize the Cheeger sets for $\mu$ in terms of specific half-spaces perpendicular to $a-b$, thereby confirming the conjectured solution to the Cheeger problem for Gaussian mixtures. Finally, we study the regime of parameters $p,a,b$ in which $\mu$ admits a unique Cheeger set.
\end{abstract}

\maketitle

\section{Introduction and results}

Let $\mu$ be a probability measure on $\R^n$ equipped with the Borel $\sigma$-algebra $\mathcal{B}(\R^n)$. Given $A \in \mathcal{B}(\R^n)$, we denote by $\mu_+(A)$ its lower Minkowski content,
$$
\mu_+(A) = \liminf_{\varepsilon \to 0} \frac{\mu(A^\varepsilon \setminus A)}{\varepsilon},
$$
where $A^\varepsilon = \{ x \in \R^n : \mathrm{dist}(A,x)<\varepsilon \}$ denotes the open $\varepsilon$-neighborhood of $A$.
When $\mu$ has a smooth density $w$ with respect to Lebesgue measure, then $\mu_+(A)$ is a natural notion of perimeter with density, and for sets $A$ with sufficiently regular boundary $\partial A$, the quantity $\mu_+(A)$ coincides with the weighted integral 
$\int_{\partial A} w \, d\mathcal H^{n-1}$ where $\mathcal H^{n-1}$ denotes the $n-1$-dimensional Hausdorff measure \cite{maggi2012sets,ambrosio2016minkowski}.
We say that $\mu$ satisfies Cheeger's inequality if there exists a constant $C>0$ such that for every $A \in \mathcal{B}(\R^n)$ one has
$$
\min \{ \mu(A), 1-\mu(A) \} \leq C \mu_+(A).
$$
The Cheeger constant of $\mu$ is defined by
$$
h_\mu = \inf_{\substack{A \in \mathcal{B}(\R^n) \\ 0<\mu(A)<1}} \frac{\mu_+(A)}{\min \{ \mu(A), 1-\mu(A) \}}.
$$
Thus $\mu$ satisfies Cheeger's inequality precisely when $h_\mu>0$, and in that case the optimal constant in Cheeger's inequality is given by $h_\mu^{-1}$.

If this infimum is attained at some set $S$ then we call $S$ a Cheeger set. The reciprocal of $h_\mu$ yields the optimal constant for the $L^1$-Poincaré inequality on $\R^n$: for every Lipschitz function $f$ one has
$$
\min_{c \in \R} \int_{\R^n} |f-c| \, d\mu \leq C_1 \int_{\R^n} |\nabla f| \, d\mu,
$$
with optimal constant $C_1 = \frac{1}{h_\mu}$. The optimal constant $C_2$ for the corresponding $L^2$-Poincaré inequality satisfies the inequality $C_2 \leq 4/h_\mu^2$ \cite{klartag2025isoperimetric}.

The problem of determining the Cheeger constant and the Cheeger sets of a measure $\mu$ is closely related to the isoperimetric problem. The latter consists in identifying, for each prescribed volume $v \in [0,1]$, those sets $A$ with $\mu(A)=v$ that minimize $\mu_+(A)$. This gives rise to the isoperimetric function of $\mu$, which is defined by
$$
I_\mu : [0,1] \to [0,\infty), \quad I_\mu(v) = \inf \{ \mu_+(A) : A \in \mathcal{B}(\R^n), \, \mu(A) = v \}.
$$
In terms of the isoperimetric function one has the equivalent representation of the Cheeger constant via
\[
h_\mu=\inf_{v\in(0,1)}\frac{I_\mu(v)}{\min\{v,1-v\}}.
\]
Solutions to the Cheeger and isoperimetric problem are only known for special types of measures, including log-concave measures. In this case, half-spaces typically turn out to be solutions to the isoperimetric problem, see \cite{klartag2025isoperimetric,bobkov1997some} and the references therein. A classical example of a log-concave probability measure is the Gaussian measure, which has density with respect to the Lebesgue measure given by
$$
\gamma : \R^n \to (0,\infty), \quad \gamma(x) = \frac{1}{(2\pi)^{\frac{n}{2}}} e^{-\frac{1}{2}|x|^2}.
$$
Denoting the Borel measure corresponding to the latter density also by $\gamma$, one has $I_\gamma = \varphi \circ \Phi^{-1}$ where
$$
\Phi(x) = \frac{1}{\sqrt{2\pi}}\int_{-\infty}^x e^{-\frac{1}{2}t^2} \, dt, \quad \varphi = \Phi'.
$$
Together with the fact that $\Phi$ is log-concave on the whole real line, one obtains that the Cheeger constant for $\gamma$ is given by $h_\gamma = \sqrt{2/\pi}$ and the Cheeger sets for $\gamma$ are half-spaces whose boundary intersect the origin \cite{bobkov1997some}. We also refer to \cite{fusco2011isoperimetric,BROCK2016375} for further results related to the isoperimetric problem of Gaussian-type measures, and to \cite{betta2008,canete2010,rosales2008,chambers2019,marini,Betta1999_2} for studies on isoperimetric problems for measures with density.

If $\mu$ is not log-concave then the problem of determining the Cheeger constant and the Cheeger sets for $\mu$ becomes mostly intractable. In the present paper we investigate the Cheeger problem for non-log-concave probability measures that are Gaussian mixtures of the form
\begin{equation}\label{eq:measure}
    \mu(x) = p\gamma(x-a) + (1-p)\gamma(x-b)
\end{equation}
where $p \in [0,1]$ and $a,b \in \R^n$. The study of the Cheeger problem for measures of the form \eqref{eq:measure} is motivated by a series of papers in the literature on isoperimetric problems, PDEs, and inverse problems: in \cite{schlichting,chafaiMalrieu} the authors establish (non-sharp) Poincaré inequalities with respect to measures of the form \eqref{eq:measure}. In \cite[Section 5, Question 6]{canete2010} it is conjectured that half-spaces are solutions to the isoperimetric problem for Gaussian mixtures of type \eqref{eq:measure}. In \cite[Theorem 3.16]{berry2018}, the latter conjecture was proved affirmatively in the one-dimensional symmetric case, i.e., $n=1$ and $p = \frac{1}{2}$. Moreover, \cite{berry2018} provides further evidence of the conjecture in the planar case $n=2$. In \cite{grohsRathmair,grohs2021stables} the Cheeger constant was used to establish stability estimates for certain non-linear inverse problems (related developments and extensions appear in the recent works \cite{alaifari2025cheeger,fuhr2025cheeger}). Here, the value $h_\mu$ serves as a quantification of the disconnectedness of $\mu$. Based on heuristics and numerical observations it is suggested in \cite{grohsRathmair} that certain sets whose boundary represent an optimal Cheeger cut positioned in between the two means $a$ and $b$ (such as a straight cut which corresponds to sets that are half-spaces) are Cheeger sets for the measure in \eqref{eq:measure}.

\subsection{Main results}
In order to state the main results of the paper, we introduce for $\alpha \in [0,1]$ and $\beta \geq 0$ the function $Q_{\alpha,\beta} : \R \to (0,\infty)$ via
$$
Q_{\alpha,\beta}(x) = \alpha \Phi(x) + (1-\alpha)\Phi(x-\beta).
$$
Moreover, we define for \(c\in\mathbb R\) and a unit vector \(\nu \in \mathbb S^{n-1}\) the half-spaces
$$
H_{c,\nu}^- := \{x\in\mathbb R^n:\ x\cdot \nu < c\}, \quad H_{c,\nu}^+ :=\{x\in\mathbb R^n:\ x\cdot \nu > c\}.
$$
Notice that if $a=b$ or $p \in \{ 0,1 \}$ in \eqref{eq:measure} then $\mu$ coincides up to translation with $\gamma$, for which the solution of the Cheeger problem is known. We therefore exclude these values in the following statement. We also remark that whenever we say two Borel sets $A,B$ satisfy $A=B$ then equality is meant up to $\mu$-null sets, i.e., the symmetric difference of $A$ and $B$ satisfies $\mu(A\Delta B) = 0$.

\begin{theorem}\label{thm:main1}
Let \(a,b\in\mathbb R^n\) with $a \neq b$, let \(p\in(0,1)\), and let $\mu$ be the measure
\[
\mu(x) = p\,\gamma(x-a)+(1-p)\,\gamma(x-b).
\]
Define $d :=|a-b|$, $m := \min \{ p,1-p \}$, and $r^* > 0$ via the equality \(Q_{m,d}(r^*)=\tfrac12\). Then the Cheeger constant of \(\mu\) is given by the minimum
\[
h_\mu=\min_{r\in[0,r^*]} \bigl(\log Q_{m,d}\bigr)'(r).
\]
If $\mathcal{O} \subset [0,r^*]$ is the set of minimizers and if $\nu := \frac{b-a}{|b-a|}$, then the following holds:
\begin{enumerate}
    \item If $p \leq \frac{1}{2}$, then a Borel set $E$ is a Cheeger set for $\mu$ if and only if $E = H_{a\cdot \nu+t,\nu}^\pm$ for some $t \in \mathcal{O}$.
    \item If $p > \frac{1}{2}$, then a Borel set $E$ is a Cheeger set for $\mu$ if and only if $E = H_{- b\cdot \nu+t,-\nu}^\pm$ for some $t \in \mathcal{O}$.
\end{enumerate}
\end{theorem}

The above theorem states that, in any dimension $n$, the Cheeger constant of $\mu$ is given by the minimum of the logarithmic derivative of $Q_{m,d}$ (a $C^\infty$ function) on the compact interval $[0,r^*]$, where $r^*$ is the median of the measure on the real line given by $m\gamma(x) + (1-m)\gamma(x-d)$. In particular, $h_\mu$ is dimension-free and depends only on the distance $d=|a-b|$ and the minimum $m=\min\{p,1-p\}$. Moreover, uniqueness of the Cheeger set up to complement corresponds to $\mathcal O$ being a singleton. In the symmetric case $p=\frac12$, this uniqueness property is achieved and an explicit formula for the Cheeger constant and the Cheeger sets is proved.

\begin{theorem}\label{thm:main2}
For every $a,b \in \R^n$, the Cheeger constant of the symmetric Gaussian mixture
$
\mu(x) = \frac12 \,\gamma(x-a)+ \frac12 \,\gamma(x-b)
$
is given by
    $$
    h_\mu = \sqrt\frac{2}{\pi} e^{-\frac{1}{8}|a-b|^2}.
    $$
    Moreover, $E$ is a Cheeger set for $\mu$ if and only if $E = H_{a\cdot \nu+t,\nu}^-$ or $E = H_{a\cdot \nu+t,\nu}^+$, where $\nu = \frac{b-a}{|b-a|}$ and $t = \frac{1}{2}|b-a|$.
\end{theorem}

Finally, we investigate uniqueness of the Cheeger set of $\mu$. The next theorem says that the Cheeger sets $H_{a\cdot \nu+t,\nu}^\pm$ (resp. $H_{- b\cdot \nu+t,-\nu}^\pm$) derived in Theorem \ref{thm:main1} correspond up to complements to at most two distinct half-spaces. Furthermore, for every $p$, if the distance $|a-b|$ between the two Gaussian centers is sufficiently large, then $\mu$ admits a unique Cheeger set.

\begin{theorem}\label{thm:main3}
    Let \(a,b\in\mathbb R^n\) with $a \neq b$, let \(p\in(0,1)\), and let $\mu=\mu_{p,a,b}$ be the measure in \eqref{eq:measure}.
Up to complements and sets of measure zero the following holds:
\begin{enumerate}
    \item The measure $\mu$ has at most two Cheeger sets.
    \item There exists $c(p)>0$ such that for every $a,b$ with $|a-b| \geq c(p)$, the measure $\mu$ has a unique Cheeger set. More explicitly, if $p\neq \frac12$, one may choose
    $$
    c(p)
    =
    1+\max\left\{
    \sqrt{8\log\frac{1}{2m^2}},
    \sqrt{4t_p^2+8\log\frac{1}{2m(1-m)}}
    \right\},
    $$
    where $m=\min\{p,1-p\}$ and
    $
    t_p=-\Phi^{-1}\left(\frac{\frac12-m}{1-m}\right)
    $.
\end{enumerate}
\end{theorem}

Given the previous statement, one may wonder if it can happen that $\mu$ has (up to complements) two distinct Cheeger sets. This situation can indeed occur (e.g. take $d=3$ and $p \approx 0.075$) and one can show that the set of parameters $(p,d) \in (0,1) \times (0,\infty)$ for which $\mu=\mu_{p,a,b}$ with $|a-b|=d$ admits two distinct Cheeger sets is meager in $(0,1) \times (0,\infty)$. Results in this direction will be the subject of a future research paper \cite{liehrFollowUp}.

\subsection{Outline and proof approach}
After reducing to the canonical mixture $\mu=p\gamma+(1-p)\gamma(\cdot-\mathbf d)$ with $\mathbf d=de_1$, the proof combines the following key ingredients: sharp isoperimetric inequalities for $\mu$ restricted to specific subclasses $\mathcal{A}$ of the Borel sets. We show that, within $\mathcal A$, sets with fixed $\mu$-volume that minimize $\mu_+$ are half-spaces orthogonal to $e_1$. This is the content of Section \ref{sec:isop}. Next, we establish a local log-concavity argument in Section \ref{sec:log_concave} yielding a monotonicity principle for the corresponding half-space Cheeger ratios with respect to $\mu$. These tools allow a global comparison across a partition of all Borel sets, which will be used to prove Theorem \ref{thm:main1} and Theorem \ref{thm:main2} in Section \ref{sec:main_results}. Finally, Section \ref{sec:33} contains the proof of Theorem \ref{thm:main3}.

\section{Isoperimetric inequalities}\label{sec:isop}

\subsection{Preliminaries}
In what follows, we work with the abbreviations $H_t^- \coloneqq H_{t,e_1}^-$ and $H_t^+ \coloneqq H_{t,e_1}^+$ where $e_1 \in \R^{n}$ denotes the first unit vector. With this abbreviation it follows that $\gamma(H_t^-) = \Phi(t)$. Moreover, we have the following identities for the $\mu$-measure and Minkowski content of half-spaces.

\begin{lemma}\label{lma:elementary}
    Let $\mu$ be the measure $\mu = p\gamma + (1-p)\gamma(\cdot-\mathbf{d})$ where $\mathbf{d}=de_1$, $d>0$, and $p \in [0,1]$. Then we have the two identities
    \begin{equation}
        \begin{split}
            \mu(H_r^-) &= p\Phi(r) + (1-p)\Phi(r-d), \\
            \mu_+(H_r^-) &= \mu_+(H_r^+) = p\varphi(r) + (1-p)\varphi(r-d).
        \end{split}
    \end{equation}
\end{lemma}
\begin{proof}
    The first identity follows directly from integration. To obtain the second identity we use that $H_r^-$ and $H_r^+$ have smooth boundary and therefore $$\mu_+(H_r^-) = \mu_+(H_r^+) = \int_{B} \left ( p\gamma+(1-p)\gamma(\cdot - \mathbf{d}) \right ) \, d\mathcal{H}^{n-1},$$ where $B = \{ x \in \R^n : x_1 = r \}$. The second identity then follows again from a direct integration.
\end{proof}

Recall that the isoperimetric function for the Gaussian measure $\gamma$ is given by $I_\gamma = \varphi \circ \Phi^{-1}$. In particular, for every \(E\in\mathcal B(\mathbb R^n)\) one has
\begin{equation}\label{eq:gii}
    \gamma_+(E)\ge I_\gamma(\gamma(E)),
\end{equation}
which is known as the Gaussian isoperimetric inequality \cite{latala2002}. This inequality was established by Borell \cite{borell1974} and independently by Sudakov and Tsirel'son \cite{sudakov1978}, and admits several later proofs and strengthenings, including Bobkov's functional inequality \cite{bobkov1996functional} and semigroup approaches (see \cite{ledoux,mosselNeeman2015, bobkov1996functional, carlen2001cases,barthe2000some} and the references therein). Moreover, it was shown in \cite{carlen2001cases} that equality in \eqref{eq:gii} is achieved only by half-spaces \(H_{t,\nu}^\pm\) satisfying $\gamma(E)=\gamma(H_{t,\nu})$, and stability issues are discussed in \cite{cianchi2011}. Notice that the Gaussian isoperimetric inequality implies that if $E$ satisfies $\mu(E) \leq \frac{1}{2}$ then the map sending $E$ to a half-space $H_{t,\nu}$ satisfying $\gamma(E)=\gamma(H_{t,\nu})$ is perimeter-decreasing and gives rise to the notion of Ehrhard symmetrization \cite{borell2003,Ehrhard1983}.
In the remainder of this section, we make use of the following elementary property of the Gaussian isoperimetric function \cite{bobkov1997some}.

\begin{lemma}\label{lma:gii_properties}
The Gaussian isoperimetric function $I_\gamma : [0,1] \to [0,\infty)$ is symmetric around $\frac{1}{2}$ and strictly increasing on $[0,\frac{1}{2}]$.
\end{lemma}

An important ingredient in the preceding proofs is the following theorem, which is known as the shift-inequality for Gaussian measure \cite{liKuelbs} (see also \cite{kuelbs1994gaussian,Oleszkiewicz,cordero2004b} for further studies on the Gaussian measure of shifts and dilates).

\begin{theorem}\label{thm:shift_inequality}
Let $E \in \mathcal{B}(\R^n)$ and let $t \in \R$ such that $\gamma(E) = \Phi(t)$. Then for every $\lambda \in \R$ and every $\nu \in \mathbb{S}^{n-1}$ one has
$$
\Phi(t - |\lambda|) \leq \gamma(E+\lambda \nu) \leq \Phi(t + |\lambda|),
$$
with equality if and only if $E$ coincides with a half-space perpendicular to $\nu$ (one for the lower bound and one for the upper bound).
\end{theorem}

A consequence of the latter Theorem is the following observation.

\begin{corollary}\label{cor:li_kuelbs_consequence}
    Let $E \in \mathcal{B}(\R^n)$ such that $\gamma(E)=\Phi(t)=\gamma(H_t^-)$. Further, let $\mathbf{d}=de_1$ with $d>0$. Then it holds that
    $$
    \gamma(E-\mathbf{d}) \geq \Phi(t-d)
    $$
    with equality if and only if $E=H_t^-$.
\end{corollary}
\begin{proof}
    Let $\lambda = -d$ and let $\nu = e_1$. Then $\gamma(E-\mathbf{d}) = \gamma(E+\lambda \nu)$. Since $\gamma(E)=\Phi(t)$, an application of Theorem \ref{thm:shift_inequality}, yields the bound
    $$
    \gamma(E-\mathbf{d}) \geq \Phi(t-d).
    $$
    Assume that equality holds in the latter inequality. Then equality holds in the lower bound of Theorem \ref{thm:shift_inequality}, so $E$ is up to null sets a half-space perpendicular to the vector $\nu$. Hence, $E=H_a^-$ or $E=H_b^+$ for some $a,b \in \R$.
    
    If $E=H_a^-$ then $\gamma(E) = \Phi(a) =\Phi(t)$ so $a=t$ and $E=H_t^-$.

    If $E=H_b^+$ then $\gamma(E) = 1-\Phi(b) = \Phi(t)$, hence $b=-t$. But then
    $$
    \gamma(E-\mathbf{d}) = \gamma(H_{-t-d}^+) = \Phi(t+d) > \Phi(t-d),
    $$
    where we used that $d>0$. This contradicts equality. It therefore follows that $E=H_a^-$ up to null sets.
\end{proof}

\subsection{Isoperimetric inequalities for $\mu$}

Given $\mathbf{d} = de_1 \in \R^n$ with $d > 0$, we define two classes of sets $\mathcal{A} \subset \mathcal{B}(\R^n)$ and $\widetilde{\mathcal{A}} \subset \mathcal{B}(\R^n)$ via
\begin{equation}\label{eq:mathcal_A}
    \begin{split}
        \mathcal A &\coloneqq \Bigl\{E\in\mathcal B(\mathbb R^n):\ \gamma(E)\ge \tfrac12,\ \gamma(E-\mathbf d)\le \tfrac12\Bigr\}, \\
        \widetilde{\mathcal A} & \coloneqq \Bigl\{E\in\mathcal B(\mathbb R^n):\ \gamma(E) < \tfrac12,\ \gamma(E-\mathbf d) > \tfrac12\Bigr\}.
    \end{split}
\end{equation}
The next statement can be regarded as a solution to the isoperimetric problem of the measure $\mu$ when restricted to sets in $\mathcal{A}$ resp. $\widetilde{\mathcal{A}}$.

\begin{theorem}\label{thm:isop_weighted}
Let $\mu$ be the measure $\mu = p\gamma + (1-p)\gamma(\cdot-\mathbf{d})$ where $\mathbf{d}=de_1$, $d>0$, and $p \in (0,1)$.
\begin{enumerate}
    \item Let \(E\in\mathcal A\) and let \(r\in\mathbb R\) be such that \(\mu(E)=\mu(H_r^-)\).
Then $r \in [0,d]$ and
\[
\mu_+(E)\ge \mu_+(H_r^-),
\]
with equality if and only if \(E = H_r^-\).
\item Let \(E\in\widetilde{\mathcal A}\) and let \(s\in\mathbb R\) be such that \(\mu(E)=\mu(H_s^+)\).
Then $s \in [0,d]$ and
\[
\mu_+(E)\ge \mu_+(H_s^+),
\]
with equality if and only if \(E = H_s^+\).
\end{enumerate}
\end{theorem}
\begin{proof}
    \textit{Claim (i).} If \(\mu_+(E)=+\infty\), the claim is trivial. Hence, we assume that \(\mu_+(E)<\infty\).
By Lemma \ref{lma:elementary} it holds that
\begin{equation}\label{eq:wE_equals}
\mu(E)=\mu(H_r^-)=p\Phi(r)+(1-p)\Phi(r-d).
\end{equation}
Since \(\gamma(E)\ge \tfrac12\) and \(\Phi\) is strictly increasing, there exists a unique \(t \ge 0\) such that
\[
\gamma(E)=\Phi(t).
\]
Applying Theorem \ref{thm:shift_inequality} with \(\lambda=-d\), and \(\nu=e_1\) gives
\[
\gamma(E-\mathbf d)=\gamma(E-de_1)\ge \Phi(t-d).
\]
By assumption, we have that $\gamma(E-\mathbf{d}) \leq \frac 12$ and since $\Phi(0)=\frac12$ and $\Phi$ is increasing we obtain $t-d \leq 0$. Hence,
$
\gamma(E) = \Phi(t) \leq \Phi(d)
$
which implies that
\begin{equation}
    \begin{split}
        \mu(H_r^-) = \mu(E) &= p \gamma(E) + (1-p) \gamma(E-\mathbf{d}) \\ 
        & \leq p \Phi(d) + (1-p)\Phi(0) = \mu(H_d^-).
    \end{split}
\end{equation}
The strict monotonicity of $r \mapsto \mu(H_r^-)$ implies $r \leq d$.

Using \(\mu(E)=p\gamma(E)+(1-p)\gamma(E-\mathbf d)\), we further obtain the inequality
\begin{equation}\label{eq:wE_lower}
\mu(E)\ge p\Phi(t)+(1-p)\Phi(t-d).
\end{equation}
Now define \(F(u):=p\Phi(u)+(1-p)\Phi(u-d)\). Then \(F\) is strictly increasing on \(\mathbb R\). Combining \eqref{eq:wE_equals}, \eqref{eq:wE_lower} and the assumption that $\mu(E) = \mu(H_r^-)$ yields \(F(r)=\mu(E)\ge F(t)\), which implies
\begin{equation}\label{eq:r_ge_t}
r \ge t\ge 0.
\end{equation}
Moreover, from \(\mu(E)=p\Phi(t)+(1-p)\gamma(E-\mathbf d)\) and \eqref{eq:wE_equals} we obtain
\begin{equation}\label{eq:gammaEd_lower}
\gamma(E-\mathbf d)
=\Phi(r-d)+\frac{p}{1-p}\bigl(\Phi(r)-\Phi(t)\bigr)
\ge \Phi(r-d),
\end{equation}
where the last inequality follows from \(r\ge t\).

Since \(E\in\mathcal A\), we have \(\gamma(E-\mathbf d)\le \tfrac12\) and therefore
\begin{equation}\label{eq:bnd}
   \Phi(r-d) \leq \gamma(E-\mathbf{d}) \leq \frac{1}{2}.
\end{equation}
According to Lemma \ref{lma:gii_properties} we have that $I_\gamma$ is strictly increasing on $[0,\frac{1}{2}]$ which gives
\[
I_\gamma(\gamma(E-\mathbf d))\ge I_\gamma(\Phi(r-d))=\varphi(r-d).
\]
Combining the latter bound with the Gaussian isoperimetric inequality yields the inequalities
\begin{align*}
\mu_+(E)
&\geq p \gamma_+(E)+(1-p)\gamma_+(E-\mathbf d)\\
&\ge pI_\gamma(\gamma(E))+(1-p)I_\gamma(\gamma(E-\mathbf d))\\
&=pI_\gamma(\Phi(t))+(1-p)I_\gamma(\gamma(E-\mathbf d))\\
&\ge p\varphi(t)+(1-p)\varphi(r-d).
\end{align*}
Finally, \eqref{eq:r_ge_t} and the fact that \(\varphi\) is strictly decreasing on \([0,\infty)\) imply
$
\varphi(t)\ge \varphi(r),
$
so
\[
\mu_+(E)\ge p\varphi(r)+(1-p)\varphi(r-d)=\mu_+(H_r^-),
\]
where the last identity is Lemma~\ref{lma:elementary}. This proves the inequality.

\smallskip
Now suppose that $E \in \mathcal{A}$ satisfies $\mu_+(E) = \mu_+(H_r^-)$. If $t$ is defined as above then the inequalities derived beforehand show that
$$
\mu_+(E) \geq p \varphi(t) + (1-p) \varphi(r-d) \geq p \varphi(r) + (1-p) \varphi(r-d) = \mu_+(H_r^-).
$$
Hence, we obtain that $t=r$. Moreover, $r$ satisfies $\mu(E) = \mu(H_r^-)$. Consequently,
\begin{equation}
    \begin{split}
        \mu(E) &= p\Phi(t) + (1-p)\gamma(E-\mathbf{d}) = p\Phi(r) + (1-p)\gamma(E-\mathbf{d}), \\
        \mu(E) &= \mu(H_r^-)  = p\Phi(r) + (1-p)\Phi(r-d).
    \end{split}
\end{equation}
This implies that $\gamma(E-\mathbf{d}) = \Phi(r-d)$. Hence, $E=H_r^-$ by Corollary \ref{cor:li_kuelbs_consequence}.

\smallskip
\textit{Claim (ii):} The second claim is analogous to the first claim by considering the measure $\tilde \mu = (1-p)\gamma + p\gamma(\cdot - \mathbf{d})$ which arises from the measure $\mu$ by applying a reflection along the hyperplane $\{ x_1 = \frac{d}{2}\}$.
\end{proof}

Define the isoperimetric function of $\mu$ restricted to the sets $\mathcal{S} \subseteq \mathcal{B}(\R^n)$ via
$$
I_{\mu,\mathcal{S}}(v) = \inf \{ \mu_+(A) : A \in \mathcal{S}, \, \mu(A) = v \}.
$$
For instance, we have seen that $I_{\gamma,\mathcal{B}(\R^n)} = \Phi' \circ \Phi^{-1}$. Theorem \ref{thm:isop_weighted} shows that if $\mu$ is the measure $\mu = p\gamma + (1-p)\gamma(\cdot-\mathbf{d})$ where $\mathbf{d}=de_1$, $d>0$, $p \in (0,1)$, and if $Q = p \Phi + (1-p)\Phi(\cdot -d)$, then $I_{\mu,\mathcal{A}}$ satisfies
$$
I_{\mu,\mathcal{A}}(v) = (Q' \circ Q^{-1})(v), \quad v \in [Q(0),Q(d)]. 
$$
A similar equality holds for the class $\widetilde{\mathcal{A}}$.

\section{Local log-concavity}\label{sec:log_concave}

Recall that the Gaussian measure is log-concave, implying that $\Phi$ is log-concave on $\R$. This global log-concavity property yields the Cheeger constant of the Gaussian $\gamma$ via the formula $h_\gamma = \sqrt{2/\pi}$. Clearly, the mixture measure $\mu$ given in \eqref{eq:measure} is not globally log-concave. However, the next Proposition identifies a compact interval where the distribution function of the measure $\mu$ in one dimension is strictly log-concave. It turns out that this local log-concavity property, when combined with the results in Section \ref{sec:isop}, is sufficient for deriving the Cheeger constant of $\mu$ for all parameters $p \in [0,1]$, $a,b \in \R^n$ and in arbitrary dimensions $n$.

\begin{proposition}\label{prop:log_concave}
    For every $a \geq 1$ and every $d \geq 0$ it holds that the function
    $$
J(x)  \coloneqq a \int_{-\infty}^x e^{-\frac{1}{2}t^2} \, dt + \int_{-\infty}^{x}  e^{-\frac{1}{2}(t-d)^2} \, dt
$$
is strictly log-concave on the interval $[0,\frac{d}{2}]$.
\end{proposition}

For the proof of Proposition \ref{prop:log_concave} we require the following Lemma.

\begin{lemma}\label{lma:ode}
    Let $f(t) = e^{-\frac{t^2}{2}}$ and $F(x) = \int_{-\infty}^x f(t) \, dt$. Then for every $x \in [0,1]$ we have the strict inequality
    $$
    \exp\left (\frac{x^2-1}{2} \right ) < x + \frac{f(x)}{F(x)}.
    $$
\end{lemma}
\begin{proof}
    Define $r(x) = \frac{f(x)}{F(x)}$ and $g(x) = \exp\left (\frac{x^2-1}{2} \right ) - x$. We want to show that the difference function
    $$
    h(x) \coloneqq r(x) - g(x)
    $$
    satisfies $h(x) > 0$ for all $x \in [0,1]$. To do so, we make the following two observations: first, a direct calculation shows that the function $r$ satisfies the differential equation
    $$
    y'(x) = -xy(x)-y(x)^2.
    $$
    Second, the function $g$ satisfies the identity
    $$
    g'(x)+xg(x)+g(x)^2 = e^{x^2-1}-1
    $$
    Since $e^{x^2-1}-1 \leq 0$ for every $x \in [0,1]$, we obtain
    $$
    g'(x) \leq -xg(x)-g(x)^2.
    $$    
    Hence, $r$ satisfies the differential equation $y'=-xy-y^2$ while $g$ satisfies the corresponding differential inequality $y' \leq -xy-y^2$.

\smallskip
Using the latter properties of $r$ and $g$, it follows that on the interval $[0,1]$ the derivative of $h$ satisfies
\begin{equation}
    \begin{split}
        h' &\geq r' + xg+g^2 \\
        & = -xr-r^2+xg+g^2 \\
        & =-x(r-g) - (r^2-g^2) \\
        & = -xh-(r+g)h
    \end{split}
\end{equation}
    Thus, for every $x \in [0,1]$ we have
    \begin{equation}\label{eq:h_prime_inequality}
        h'(x) + (x+r(x)+g(x))h(x) \geq 0.
    \end{equation}
    Define
    $$
    k(x) = \int_0^x (t+r(t)+g(t)) \, dt, \quad I(x)=e^{k(x)}. 
    $$
    Then
    $$
    (Ih)' = I(h'+k'h),
    $$
    which shows in combination with inequality \eqref{eq:h_prime_inequality} that $Ih$ is non-decreasing on $[0,1]$. Observe further that
    $$
    h(0) = r(0) - g(0) = \sqrt{\frac{2}{\pi}} - e^{-\frac{1}{2}} > 0.
    $$
    Since $I$ is strictly positive, this implies $I(0)h(0) > 0$. The property that $Ih$ is non-decreasing implies that $Ih > 0$ on $[0,1]$. Since $I > 0$ it follows that $h >0$ which yields the claim.
\end{proof}

We are prepared to prove Proposition \ref{prop:log_concave}.

\begin{proof}[Proof of Proposition \ref{prop:log_concave}]
    Fix $a \geq 1$ and $d \geq 0$. It suffices to show that for every $x \in [0,\frac{d}{2}]$ it holds that $(\log J)''(x) < 0$. To do so, we define
    $$
    f(t) = e^{-\frac{t^2}{2}}, \quad F(x) = \int_{-\infty}^x f(t) \, dt, \quad T(x) = \int_x^\infty f(t) \, dt.
    $$
    With the latter definitions it holds that
    $$
    J(x) = aF(x) + F(x-d)
    $$
    Moreover, define $N = JJ''-(J')^2$. Then $(\log J)''(x) < 0$ is equivalent to the condition that $N(x) < 0$.

    It will be convenient to introduce the linear function
    $$
    z(x) = d-x.
    $$
    With this function, the condition that $x \in [0,\frac{d}{2}]$ can be expressed via the condition that $z(x) \geq x \geq 0$. Moreover, it is easy to see that the function $J$ can be expressed via
    $$
    J(x) = aF(x) + T(z(x)).
    $$
    The first and second derivatives of the function $x \mapsto T(z(x))$ satisfies
    $$
    \frac{d}{dx} T(z(x)) = f(z(x)), \quad \frac{d^2}{dx^2} T(z(x)) = z(x)f(z(x)).
    $$
    Therefore,
    $$
    J'(x) = a f(x) + f(z(x)), \quad J''(x)=-axf(x) + z(x) f(z(x)).
    $$
    We now write $N(x)$ as a quadratic in $a$: using the definition of $J$ and the expression derived for $J'$ and $J''$ we obtain that
    $$
    N(x) = a^2A(x)+aB(x)+C(x)
    $$
    where the functions $A,B,C$ are defined via
    \begin{equation}
        \begin{split}
            A(x) &= -x F(x) f(x) - f(x)^2, \\
            B(x) &= F(x)z(x)f(z(x)) - 2 f(x)f(z(x)) - xf(x)T(z(x)), \\
            C(x) &= T(z(x)) z(x) f(z(x)) - f(z(x))^2.
        \end{split}
    \end{equation}
    Consider the following three inequalities between the functions $A,B$ and $C$:
    \begin{enumerate}
        \item[(i)] $A(x) \leq 0$ for every $x \geq 0$,
        \item[(ii)] $C(x) \leq 0$ for every $x$ so that $z(x) \geq 0$,
        \item[(iii)] $B(x) < -A(x)$ for every $x$ satisfying $0\leq x \leq z(x)$.
    \end{enumerate}
    If (i), (ii) and (iii) hold and $0\leq x \leq z(x)$, then using $a \geq 1$ we obtain
    $$
    N(x) = a^2A(x) + aB(x) + C(x) < a(a-1)A(x) \leq 0,
    $$
    showing that $J$ is strictly log-concave on the interval $[0,\frac{d}{2}]$. Thus, it suffices to show the validity of the inequalities (i), (ii) and (iii).

    \textit{Proof of (i).} Since $xF(x)f(x) \geq 0$ for every $x \geq 0$, the inequality stated in (i) follows immediately.

    \textit{Proof of (ii).} We have that
    $$
    C(x) = f(z(x)) ( T(z(x))z(x) - f(z(x)) ). 
    $$
    Hence, it suffices to show that
    $$
    f(z(x)) - T(z(x))z(x) \geq 0.
    $$
    To do so, we observe that
    $$
    \int_z^\infty t f(t) \, dt = f(z).
    $$
    Therefore,
    $$
    f(z(x)) - T(z(x))z(x) = \int_{z(x)}^\infty (t-z(x)) f(t) \, dt \geq 0.
    $$
    
    \textit{Proof of (iii).} By definition of the function $B$ and the positivity of $T$ and $f$, we have that
    $$
    B(x) \leq F(x) z(x) f(z(x)).
    $$
    It therefore suffices to prove that
    $$
    F(x) z(x) f(z(x)) \leq x F(x) f(x) + f(x)^2.
    $$
    Dividing both sides by $F(x)f(x)>0$ we obtain
    $$
    \frac{z(x)f(z(x))}{f(x)} \leq x + r(x),
    $$
    where $r$ denotes the ratio
    $$
    r(x) = \frac{f(x)}{F(x)}.
    $$
    Consider $h(u) = u f(u)$. Then $h'(u) = f(u) (1-u^2)$ which implies that $h$ increases on $[0,1]$ and decreases on $[1,\infty)$ and it attains its maximum at $t=1$. Now fix $x \geq 0$ and $z(x) \geq x$. If $x \geq 1$ then $z(x) \geq x \geq 1$ which yields
    $$
    z(x) f(z(x)) = h(z(x)) \leq h(x) = x f(x),
    $$
    where we used that $h$ is decreasing on $[1,\infty)$. This gives
    $$
    \frac{z(x)f(z(x))}{f(x)} \leq x \leq x + r(x).
    $$
    It remains to consider the case when $z(x) \geq x \geq 0$ and $x \in [0,1]$. Using that $x=1$ is the global maximum of $h$ on $[0,\infty)$ we obtain that
    $$
    \frac{z(x)f(z(x))}{f(x)} = \frac{h(z(x))}{f(x)} \leq \frac{h(1)}{\varphi(x)} = \exp\left(\frac{x^2-1}{2}\right).
    $$
    Finally, using Lemma \ref{lma:ode}, we have
    $$
    \exp\left(\frac{x^2-1}{2}\right) \leq x + r(x)
    $$
    and this completes the proof of inequality (iii).
\end{proof}

As an application of Proposition \ref{prop:log_concave}, we obtain the following Corollary, which will be frequently applied in the upcoming section.

\begin{corollary}\label{cor:strictly_decreasing}
    Let $p \in (0,\frac{1}{2}]$ and let $\mathbf{d} = de_1$ with $d >0$. Further, let $\mu$ be the measure $\mu = p\gamma + (1-p)\gamma(\cdot - \mathbf{d})$. Then the function
    $$
    t \mapsto \frac{\mu_+(H_{d-t}^+)}{\mu(H_{d-t}^+)}
    $$
    is strictly decreasing on $[0,\frac{d}{2}]$.
\end{corollary}
\begin{proof}
    Define the measure $\tilde \mu = (1-p)\gamma + p \gamma(\cdot - \mathbf{d})$. Then we have
    $$
    \mu(H_{d-t}^+) = \tilde \mu(H_t^-), \quad \mu_+(H_{d-t}^+) = \tilde \mu_+(H_t^-).
    $$
    Therefore,
    $$
    \frac{\mu_+(H_{d-t}^+)}{\mu(H_{d-t}^+)} = \frac{\tilde\mu_+(H_{t}^-)}{\tilde\mu(H_{t}^-)} = (\log J)'(t)
    $$
    where $J$ is given as in Proposition \ref{prop:log_concave} with $a=\frac{1-p}{p} \geq 1$. Since $J$ is strictly log-concave on $[0,\frac d2]$ it follows that $(\log J)'$ is strictly decreasing on $[0,\frac d2]$.
\end{proof}

\section{Proof of Theorem \ref{thm:main1} and Theorem \ref{thm:main2}}\label{sec:main_results}

We start with a technical observation that concerns the location of the median $r^*$, defined via the equality $Q_{p,d}(r^*)=\frac12$, relative to the parameter $d$. Recall that for $\alpha \in [0,1]$ and $\beta \geq 0$, the function $Q_{\alpha,\beta}$ is defined by $Q_{\alpha,\beta}(x) = \alpha \Phi(x) + (1-\alpha)\Phi(x-\beta)$.

\begin{lemma}\label{lma:rstar_location}
Let \(p\in (0,\tfrac12 ]\), let \(d > 0\), and let $r^*$ be the unique value such that $Q_{p,d}(r^*)= \frac12$. Then $r^* \in [ \frac d2,d)$.
\end{lemma}
\begin{proof}
    Using that $\Phi(-\frac d2) = 1-\Phi(\frac d2)$ as well as $0<p\leq\frac{1}{2}$ and $\Phi(x)>\frac 12$ for every $x > 0$ it follows that
    $$
    Q(\tfrac d2) - \tfrac{1}{2} = (\Phi(\tfrac d 2) - \tfrac 12)(2p-1) \leq 0
    $$
    and hence $Q(\frac d2) \leq \frac 12$. Moreover, we have
    $$
    Q(d) = p(\Phi(d)-\tfrac 12) + \tfrac{1}{2} > \tfrac{1}{2}.
    $$
    Since $Q$ is continuous and strictly increasing, we obtain $r^* \in [\frac d2,d)$.
\end{proof}

We recall that the Cheeger constant for the Gaussian $\gamma$ satisfies $h_\gamma = \sqrt{2/\pi}$. Therefore, Cheeger's inequality for the Gaussian $\gamma$ reads
$$
\gamma_+(E)\ge \sqrt{\frac{2}{\pi}}\min\{\gamma(E),1-\gamma(E)\}.
$$
Now introduce for $\mathbf{d}=de_1$ with $d>0$ the classes of sets $\mathcal{X}$ and $\widetilde{\mathcal{X}}$ via
\begin{equation}
    \begin{split}
        \mathcal{X} &= \{ E \in \mathcal{B}(\R^n) : \gamma(E) < \tfrac{1}{2} ,\, \gamma(E-\mathbf{d}) \leq \tfrac{1}{2} \}, \\
        \widetilde{\mathcal{X}} &= \{ E \in \mathcal{B}(\R^n) : \gamma(E) \geq \tfrac12 ,\, \gamma(E-\mathbf{d}) > \tfrac{1}{2} \}.
    \end{split}
\end{equation}
Together with $\mathcal{A}$ and $\widetilde{\mathcal{A}}$ defined in \eqref{eq:mathcal_A}, it follows that the Borel sets $\mathcal{B}(\R^n)$ split into a union of the form
$$
\mathcal{B}(\R^n) = \mathcal{A} \cup \widetilde{\mathcal{A}} \cup \mathcal{X} \cup \widetilde{\mathcal{X}}.
$$
For $E \in \mathcal{B}(\R^n)$ with $0<\mu(E)<1$ we call the quantity $\frac{\mu_+(E)}{\min \{ \mu(E),1-\mu(E) \}}$ the Cheeger ratio of $E$ with respect to the measure $\mu$.

\begin{theorem}\label{thm:cheeger_mixture}
Let \(p\in(0,\tfrac12]\), \(d > 0\), and let \(r^* > 0\) be the unique value such that \(Q(r^*) = Q_{p,d}(r^*)=\tfrac12\).
Then the Cheeger constant of $\mu = p\gamma + (1-p)\gamma(\cdot - \mathbf{d})$ satisfies
\[
h_\mu=\min_{r\in[0,r^*]} \bigl(\log Q \bigr)'(r).
\]
If $\mathcal{O} \subset [0,r^*]$ denotes the set of minimizers, then $E$ is a Cheeger set for $\mu$ if and only if $E=H_t^\pm$ for some $t \in \mathcal{O}$.
\end{theorem}
\begin{proof}
Define the minimum $\sigma>0$ by $\sigma = \min_{r\in[0,r^*]} \bigl(\log Q \bigr)'(r)$. We start by showing that $h_\mu \geq \sigma$. To do so, we apply a comparison argument which is split into Case 1 to Case 4.2.

\smallskip
    \textit{Case 1.} Let $E \in \mathcal{X}$. Since $\gamma(E) < \frac12$ and $\gamma(E-\mathbf{d}) \leq \frac{1}{2}$ it follows that $\mu(E) < \frac 12$. Hence,
$$
\min \{ \mu(E),1-\mu(E) \} = \mu(E) = p \gamma(E) + (1-p)\gamma(E-\mathbf{d}).
$$
Moreover, Cheeger's inequality for the Gaussian implies that
$$
\gamma_+(E) \geq \sqrt{\frac{2}{\pi}} \gamma(E), \quad \gamma_+(E-\mathbf{d}) \geq \sqrt{\frac{2}{\pi}} \gamma(E-\mathbf{d}).
$$
Hence, we have
$$
\frac{\mu_+(E)}{\min \{ \mu(E),1-\mu(E) \}} = \frac{\mu_+(E)}{\mu(E)} \geq \frac{p\gamma_+(E) + (1-p) \gamma_+(E-\mathbf{d})}{p \gamma(E) + (1-p)\gamma(E-\mathbf{d})} \geq \sqrt{\frac{2}{\pi}}.
$$
Consider the half-space $H_{r^*}^- \in \mathcal{A}$. Then
$$
\frac{\mu_+(H_{r^*}^-)}{\min \{ \mu(H_{r^*}^-),1-\mu(H_{r^*}^-) \}} = 2 ( p\varphi(r^*) + (1-p)\varphi(r^*-d) ). 
$$
Since $r^* > 0$ by Lemma \ref{lma:rstar_location}, it follows that
\begin{equation}\label{eq:strict}
    \frac{\mu_+(H_{r^*}^-)}{\min \{ \mu(H_{r^*}^-),1-\mu(H_{r^*}^-) \}} < 2 \varphi(0) = \sqrt{\frac{2}{\pi}}.
\end{equation}
Thus every Cheeger ratio of $E \in \mathcal X$ is strictly larger than the Cheeger ratio of the set $H_{r^*}^-$ which belongs to $\mathcal{A}$, implying that no set in $\mathcal{X}$ is a Cheeger set.

\textit{Case 2.} Let $E \in \widetilde{\mathcal{X}}$. Similar to Case 1, we first observe that $\mu(E) > \frac{1}{2}$ and apply Cheeger's inequality for the Gaussian measure to obtain
$$
\gamma_+(E) \geq \sqrt{\frac{2}{\pi}} (1- \gamma(E)), \quad \gamma_+(E-\mathbf{d}) \geq \sqrt{\frac{2}{\pi}} (1-\gamma(E-\mathbf{d})).
$$
An analogous argument as in Case 1 yields
$$
\frac{\mu_+(E)}{\min \{ \mu(E),1-\mu(E) \}} \geq \sqrt{\frac{2}{\pi}}.
$$
Together with the strict inequality \eqref{eq:strict}, it follows that the Cheeger ratio of every $E \in \widetilde{\mathcal X}$ is strictly larger than the Cheeger ratio of $H_{r^*}^- \in \mathcal{A}$, thereby showing that no set in $\widetilde{\mathcal X}$ is a Cheeger set for $\mu$.

\textit{Case 3.1.} Suppose that $E \in \mathcal{A}$ satisfies $\mu(E) \leq \frac 12$. Let $r \in \R$ be the unique value such that $\mu(E) = \mu(H_r^-) = Q(r)$. From Theorem \ref{thm:isop_weighted} it follows that $r \geq 0$ and since $\mu(E) \leq \frac 12$ we have $r \leq r^*$. Applying Theorem \ref{thm:isop_weighted} implies
$$
\mu_+(E) \geq \mu_+(H_r^-) = Q'(r).
$$
Therefore, we obtain
$$
\frac{\mu_+(E)}{\min \{ \mu(E),1-\mu(E) \}} = \frac{\mu_+(E)}{\mu(E)} \geq \frac{Q'(r)}{Q(r)} = (\log Q)'(r) \geq \sigma.
$$

\textit{Case 3.2.} Suppose that $E \in \mathcal{A}$ satisfies $\mu(E) > \frac 12$. Again, let $r \in \R$ such that $\mu(E) = \mu(H_r^-) = Q(r)$. Since $\mu(E) > \frac12$ it follows that $r^* < r$ and from Theorem \ref{thm:isop_weighted} we obtain $r \leq d$, hence $r \in (r^*,d]$. Moreover, Theorem \ref{thm:isop_weighted} shows that $\mu_+(E) \geq \mu_+(H_r^-) = Q'(r)$. Hence,
$$
\frac{\mu_+(E)}{\min \{ \mu(E),1-\mu(E) \}} = \frac{\mu_+(E)}{1-\mu(E)} = \frac{\mu_+(E)}{\mu(H_r^+)} \geq \frac{\mu_+(H_r^-)}{\mu(H_r^+)} = \frac{\mu_+(H_r^+)}{\mu(H_r^+)}.
$$
By Lemma \ref{lma:rstar_location} we have $(r^*,d] \subset [\frac{d}{2},d]$. Applying Corollary \ref{cor:strictly_decreasing} shows
$$
\frac{\mu_+(H_r^+)}{\mu(H_r^+)} > \frac{\mu_+(H_{r^*}^+)}{\mu(H_{r^*}^+)} = \frac{\mu_+(H_{r^*}^-)}{\min \{ \mu(H_{r^*}^-),1-\mu(H_{r^*}^-) \}} =  (\log Q)'(r^*).
$$
Similarly as in Case 1 and Case 2, this shows that the Cheeger ratio of $H_{r^*}^-$ is strictly smaller than the Cheeger ratio of every set $E \in \mathcal{A}$ satisfying $\mu(E) > \frac 12$.

\textit{Case 4.1.} Let $E \in \widetilde{\mathcal{A}}$ satisfy $\mu(E) < \frac 12$. Let $s \in \R$ be the unique value such that $\mu(E) = \mu(H_s^+)$. Then Theorem \ref{thm:isop_weighted} yields
$$
\frac{\mu_+(E)}{\min \{ \mu(E),1-\mu(E) \}} = \frac{\mu_+(E)}{\mu(E)} = \frac{\mu_+(E)}{\mu(H_s^+)} \geq \frac{\mu_+(H_s^+)}{\mu(H_s^+)}.
$$
Using the assumption $\mu(E) < \frac 12$ together with Theorem \ref{thm:isop_weighted} and Lemma \ref{lma:rstar_location} it follows that $s \in (r^*,d] \subseteq [\frac{d}{2},d]$. Therefore, Corollary \ref{cor:strictly_decreasing} implies
$$
\frac{\mu_+(H_s^+)}{\mu(H_s^+)} > \frac{\mu_+(H_{r^*}^+)}{\mu(H_{r^*}^+)} = \frac{\mu_+(H_{r^*}^-)}{\mu(H_{r^*}^-)} = (\log Q)'(r^*) \geq \sigma.
$$

\textit{Case 4.2.} Finally, let $E \in \widetilde{\mathcal{A}}$ such that $\mu(E) \geq \frac 12$. If $s$ is the unique value such that $\mu(E) = \mu(H_s^+)$ then the condition $\mu(E) > \frac12$ together with Theorem \ref{thm:isop_weighted} implies $s \in [0,r^*]$. Thus,
\begin{equation}
    \begin{split}
        \frac{\mu_+(E)}{\min \{ \mu(E),1-\mu(E) \}} &= \frac{\mu_+(E)}{1-\mu(H_s^+)} = \frac{\mu_+(E)}{\mu(H_s^-)} \geq \frac{\mu_+(H_s^+)}{\mu(H_s^-)} = \frac{\mu_+(H_s^-)}{\mu(H_s^-)} \\
        & = (\log Q)'(s) \geq \sigma.
    \end{split}
\end{equation}

\smallskip
Combining the above cases, it follows that $h_\mu \geq \sigma$. Moreover, if $t \in \mathcal{O}$ is a minimizer then the set $H_t^-$ belongs to $\mathcal{A}$ and both $H_t^-$ and $H_t^+$ are Cheeger sets for $\mu$. This shows that $h_\mu = \sigma$.

\smallskip
It remains to show that if $E$ is a Cheeger set for $\mu$ then there exists $t \in \mathcal{O}$ such that $E = H_t^-$ or $E=H_t^+$. To show this, observe that if $E$ is a Cheeger set then it follows from the above that the relevant cases are Case 3.1 and Case 4.2, i.e., $E \in \mathcal{A}$ with $\mu(E) \leq \frac12$ or $E \in \widetilde{\mathcal{A}}$ with $\mu(E) \geq \frac12$. In the first case, it follows from the inequalities discussed in Case 3.1 that $\mu_+(E) = \mu_+(H_t^-)$ for some $t \in \mathcal{O}$. Theorem \ref{thm:isop_weighted} therefore shows that $E=H_t^-$. Using an analogous argument for the second case implies that if $E \in \widetilde{\mathcal{A}}$ with $\mu(E) \geq \frac12$ then $\mu_+(E) = \mu_+(H_s^+)$ for some $s \in \mathcal{O}$ and Theorem \ref{thm:isop_weighted} yields $E=H_s^+$. Consequently, we have
$
E \in \{ H_t^\pm : t \in \mathcal{O} \}.
$
\end{proof}

\begin{proof}[Proof of Theorem \ref{thm:main1}]
    By applying a reflection along the hyperplane $\{ x_1 = \tfrac{d}{2} \}$, we obtain an analogous statement for Theorem \ref{thm:cheeger_mixture} in the case where $p\geq\tfrac12$. If $a,b \in \R^n$ satisfy $|a-b|=d$ then the general Gaussian mixture $\mu = p\gamma(\cdot - a) + (1-p) \gamma(\cdot -b)$ arises from the special measure $\tilde \mu = p\gamma + (1-p) \gamma(\cdot -\mathbf{d})$ via a rigid transformation. Theorem \ref{thm:main1} therefore follows from Theorem \ref{thm:cheeger_mixture} by applying a rigid transformation that sends the vector $\mathbf{d}$ to $b-a$.
\end{proof}

\begin{proof}[Proof of Theorem \ref{thm:main2}]
    In the case where $p = \frac12$, we have that the value $r^*$ for the measure $\mu = \frac{1}{2}\gamma + \frac12 \gamma(\cdot - \mathbf{d})$ is given by $r^* = \frac{d}{2}$. It follows from an equivalent formulation of Corollary \ref{cor:strictly_decreasing}, that $Q=Q_{\frac12,d}$ is strictly log-concave on $[0,\frac d2]$. Hence, $(\log Q)'$ is strictly decreasing on that interval and the unique minimizer is $\mathcal{O}= \{ \frac d 2\}$. Hence,
$$
h_\mu = (\log Q)'(\tfrac d2) = 2\varphi(\tfrac{d}{2}) = \sqrt{\frac{2}{\pi}} e^{-\frac18 d^2}
$$
and the unique Cheeger sets are given by $H_{\frac d2}^\pm$. Theorem \ref{thm:main2} then follows with the same reasoning as Theorem \ref{thm:main1} followed from Theorem $\ref{thm:cheeger_mixture}$.
\end{proof}

\section{Proof of Theorem \ref{thm:main3}}\label{sec:33}

\begin{lemma}\label{lma:number_of_zeros}
    Let $p \in (0,\frac12]$, let $d >0$, and let $f \coloneqq (\log Q_{p,d})'$. Then $f'$ has at most two zeros on $\R$.
\end{lemma}
\begin{proof}
Let $Q = Q_{p,d}$ and let $q=1-p$. Consider the function $h$ defined by $h=\frac{Q''}{Q'}$. A direct computation shows that $f'=f\big(h-f\big)$. Since $f(x)>0$ for all $x \in \R$, the zeros of $f'$ coincide with the zeros of $f-h$. Further, define
\[
F \colonequals Q\big(f-h\big).
\]
Since $Q(x)>0$ for all $x\in\R$, the functions $F$ and $f-h$ have the same zeros, and $f'$ and $f-h$ have the same zeros as well.
Thus it suffices to show that $F$ has at most two zeros on $\R$.

Using that $Q \cdot f=Q'$, we can write
$
F(x)=Q'(x)-Q(x)\,h(x).
$
Differentiating and using $h=Q''/Q'$ gives
\begin{equation}\label{eq:Fprime}
F'(x)=-Q(x)\,h'(x).
\end{equation}
Next, we now introduce the ratio
\[
k(x)\colonequals \frac{\varphi(x-d)}{\varphi(x)}
=\exp\Big(d x-\frac{d^2}{2}\Big).
\]
It follows from a calculation that the derivative of $h$ can be expressed in terms of $p,q,d$ and $k(x)$ via
\begin{equation}\label{eq:hprime-explicit}
h'(x)=-1+\frac{p q\,d^2\,k(x)}{(p+q\,k(x))^2}.
\end{equation}
Consider the function $S:(0,\infty)\to(0,\infty)$ and its derivative $S'$ given by
\[
S(u)\colonequals \frac{u}{(p+q u)^2}, \quad S'(u)=\frac{p-q u}{(p+q u)^3}.
\]
Then $S$ is strictly increasing on $(0,p/q)$ and strictly decreasing on $(p/q,\infty)$,
with a unique maximum at $u=p/q$, satisfying
\[
\max_{u>0} S(u)=S(p/q)=\frac{1}{4pq}.
\]
Since $k:\R\to(0,\infty)$ is a strictly increasing bijection, the map
$x\mapsto S(k(x))$ is strictly increasing up to the unique point $x_0$
where $k(x_0)=p/q$ and strictly decreasing thereafter. Hence the second term
in \eqref{eq:hprime-explicit} satisfies
\[
\max_{x\in\R}\frac{p q\,d^2\,k(x)}{(p+q\,k(x))^2}
=pq\,d^2 \max_{u>0}S(u)=\frac{d^2}{4}.
\]
Consequently we have the following:
\begin{itemize}
\item If $d<2$, then $h'(x)<0$ for all $x$.
\item If $d=2$, then $h'(x)\le 0$ for all $x$, and $h'(x)=0$ iff $x=x_0$.
\item If $d>2$, then $h'$ has exactly two distinct zeros $\alpha<\beta$ and
\begin{equation}\label{eq:hprime-sign}
h'(x)<0\ (x<\alpha),\qquad
h'(x)>0\ (\alpha<x<\beta),\qquad
h'(x)<0\ (x>\beta),
\end{equation}
\end{itemize}
where we used that $h'(x)\to -1$ as $x\to\pm\infty$.

\smallskip
Next we observe that
\begin{equation}\label{eq:Fminusinf}
\lim_{x \to -\infty} F(x)=0
\end{equation}
Moreover, from \eqref{eq:Fprime}, $F'(x)=-Q(x)h'(x)$ and $Q(x)>0$, so $F'$ has the opposite sign of $h'$. According to the above discussion on the sign-behavior of $h'$, there exists $\alpha \in \R$ such that $F$ is strictly increasing on $(-\infty,\alpha)$ (for every $d$). From this and the limit in $\eqref{eq:Fminusinf}$, it follows that $F>0$ on $(-\infty,\alpha)$.

Combining the latter with the sign information above we obtain:

\begin{itemize}
\item If $d<2$, then $h'<0$ everywhere, hence $F'>0$ everywhere and $F$ is strictly increasing. Thus, $F(x)>0$ for all $x\in\R$, hence $F$ has no zeros.
\item If $d=2$, then $h'\le 0$ everywhere, so $F'\ge 0$ and $F$ is nondecreasing. Moreover,
$F'(x)>0$ except at most one point, hence $F$ is strictly increasing. Thus, $F(x)>0$ for all $x\in\R$.
\item If $d>2$, let $\alpha<\beta$ be as in \eqref{eq:hprime-sign}. Then $F'$ is positive on
$(-\infty,\alpha)$, negative on $(\alpha,\beta)$ and positive on $(\beta,\infty)$, so $F$ has at most two zeros on $\R$.
\end{itemize}
\end{proof}

\begin{proof}[Proof of Theorem \ref{thm:main3}(1)]
    Let $f \coloneqq (\log Q)'$ and let $r^* > 0$ be the unique value such that $Q(r^*)=\frac12$. Suppose $f$ had three distinct minimizers in $[0,r^*]$,
say $0\le a<b<c\le r^*$. Then $b\in(0,r^*)$ is an interior minimizer,
hence $f'(b)=0$. Also, since $f(a)=f(b)$ and $f$ is continuous on $[a,b]$
and differentiable on $(a,b)$, Rolle's theorem yields
$\xi_1\in(a,b)$ such that $f'(\xi_1)=0$. Similarly, from $f(b)=f(c)$
we get $\xi_2\in(b,c)$ with $f'(\xi_2)=0$. Thus $f'$ would have three
distinct zeros $\xi_1<b<\xi_2$, a contradiction.
\end{proof}

\begin{proof}[Proof of Theorem \ref{thm:main3}(2)]
The uniqueness of the Cheeger set for $p=\frac12$ was discussed in Theorem \ref{thm:main2}. By symmetry, it suffices to consider the case $p \in (0,\frac12)$.

\smallskip
Let $q=1-p$, let $Q=Q_{p,d}$, let $f=f_{p,d}=(\log Q)'$, and let $r^*=r^*(p,d)$ be defined by $Q(r^*)=\frac12$. Since $p<\frac12$, we have
\begin{equation}\label{eq:iddd}
    Q(\tfrac d2)-\tfrac12
    =
    \bigl(\Phi(\tfrac d2)-\tfrac12\bigr)(2p-1)
    <0.
\end{equation}
Thus $r^*>\frac d2$. Moreover, $Q'_{p,d}(\frac d2)=\varphi(\frac d2)$, and from \eqref{eq:iddd} we obtain
$$
Q_{p,d}(\tfrac d2)
=
q+(2p-1)\Phi(\tfrac d2)
>
q+(2p-1)=p,
$$
where we used $2p-1<0$. Therefore,
\begin{equation}\label{eq:f-at-half-upper}
f_{p,d}(d/2)
=
\frac{Q'_{p,d}(d/2)}{Q_{p,d}(d/2)}
\leq
\frac{\varphi(d/2)}{p}
=
\frac{1}{p\sqrt{2\pi}}e^{-\frac{d^2}{8}}.
\end{equation}
Also, since $Q(0)\leq \frac12$, we have
$$
f_{p,d}(0)
=
\frac{Q'_{p,d}(0)}{Q_{p,d}(0)}
\geq
2p\varphi(0)
=
\frac{2p}{\sqrt{2\pi}}.
$$

\smallskip
We next bound $f_{p,d}(r^*)$ from below. Define the right-tail
$$
\varepsilon_d=1-\Phi(r^*)\in(0,1).
$$
Using $Q_{p,d}(r^*)=\frac12$, we obtain
$$
q\Phi(r^*-d)
=
\frac12-p\Phi(r^*)
=
\frac12-p(1-\varepsilon_d)
=
\Bigl(\frac12-p\Bigr)+p\varepsilon_d.
$$
Hence
\begin{equation}\label{eq:Phi-shift-identity}
\Phi(r^*-d)
=
\frac{\frac12-p}{q}+\frac{p}{q}\varepsilon_d
=
\alpha_p+\frac{p}{q}\varepsilon_d,
\end{equation}
where
$$
\alpha_p=\frac{\frac12-p}{q}\in(0,\tfrac12).
$$
Let
$$
t_p=-\Phi^{-1}(\alpha_p)>0.
$$
Then $\Phi(r^*-d)\geq \alpha_p=\Phi(-t_p)$, and therefore $r^*-d\geq -t_p$. Since $r^*\leq d$ by Lemma \ref{lma:rstar_location}, it follows that $\varphi(r^*-d)\geq \varphi(t_p)$. Consequently,
$$
f_{p,d}(r^*)
=
2Q'_{p,d}(r^*)
\geq
2q\varphi(r^*-d)
\geq
2q\varphi(t_p)
=
\frac{2q}{\sqrt{2\pi}}e^{-\frac{t_p^2}{2}}.
$$

Combining these estimates shows that $f_{p,d}(d/2)<f_{p,d}(0)$ and $f_{p,d}(d/2)<f_{p,d}(r^*)$ provided that
$$
d>
\max\left\{
\sqrt{8\log\frac{1}{2p^2}},
\sqrt{4t_p^2+8\log\frac{1}{2pq}}
\right\}.
$$
Thus one may choose
$$
c(p)
=
1+\max\left\{
\sqrt{8\log\frac{1}{2p^2}},
\sqrt{4t_p^2+8\log\frac{1}{2pq}}
\right\}.
$$
For every $d\geq c(p)$, the point $d/2\in[0,r^*]$ has strictly smaller value of $f_{p,d}$ than both endpoints $0$ and $r^*$; hence neither endpoint can be a minimizer.

\smallskip
Since both $0$ and $r^*$ are not minimizers, every minimizer lies in the open interval $(0,r^*)$. If there were two minimizers, then $f'$ would have two distinct zeros in $(0,r^*)$ and, by Rolle's theorem, also a third zero, contradicting Lemma \ref{lma:number_of_zeros}.
\end{proof}

\section*{Acknowledgments}

The author is grateful to the Azrieli Foundation for the award of an Azrieli Fellowship and acknowledges the support of this research by ISF Grant No.~854/25 and ISF Grant No.~1044/21. Moreover, the author thanks the anonymous referee for valuable comments.

\bibliographystyle{abbrv}
\bibliography{bibfile}

@article{klartag2025isoperimetric,
  title={Isoperimetric inequalities in high-dimensional convex sets},
  author={Klartag, Boaz and Lehec, Joseph},
  journal={Bull. Amer. Math. Soc.},
  volume={62},
  number={4},
  pages={575--642},
  year={2025}
}

@article{carlen2001cases,
  title={{On the cases of equality in Bobkov's inequality and Gaussian rearrangement}},
  author={Carlen, E. A. and Kerce, C.},
  journal={Calc. Var. Partial Differential Equations},
  volume={13},
  pages={1--18},
  year={2001},
  publisher={Springer}
}

@InProceedings{ledoux,
author="Ledoux, Michel",
editor="Eberlein, Ernst
and Hahn, Marjorie
and Talagrand, Michel",
title="A Short Proof of the Gaussian Isoperimetric Inequality",
booktitle="High Dimensional Probability",
year="1998",
publisher="Birkh{\"a}user Basel",
address="Basel",
pages="229--232",
isbn="978-3-0348-8829-5"
}

@article{sudakov1978,
author = {Sudakov, V.N. and Tsirel'son, B.S.},
title = "Extremal properties of half-spaces for spherically invariant measures",
journal = "J. Math. Sci.",
volume = "9",
pages = "9--18",
year = "1978",
url = "https://doi.org/10.1007/BF01086099",
doi = "10.1007/BF01086099"
}

@article{borell2003,
title = {{The Ehrhard inequality}},
journal = "C. R. Math. Acad. Sci. Paris",
volume = "337",
number = "10",
pages = "663 - 666",
year = "2003",
issn = "1631-073X",
doi = "https://doi.org/10.1016/j.crma.2003.09.031",
url = "http://www.sciencedirect.com/science/article/pii/S1631073X03004461",
author = "C. Borell"
}

@article{cordero2004b,
  title={The (B) conjecture for the Gaussian measure of dilates of symmetric convex sets and related problems},
  author={Cordero-Erausquin, Dario and Fradelizi, Matthieu and Maurey, Bernard},
  journal={J. Funct. Anal.},
  volume={214},
  number={2},
  pages={410--427},
  year={2004},
  publisher={Elsevier}
}

@article{Ehrhard1983,
author = {Ehrhard, Antoine},
journal = {Math. Scand.},
pages = {281-301},
title = {Symétrisation dans l'espace de Gauss.},
url = {http://eudml.org/doc/166873},
volume = {53},
year = {1983}
}

@inproceedings{barthe2000some,
  title={Some remarks on isoperimetry of Gaussian type},
  author={Barthe, Franck and Maurey, Bernard},
  booktitle={Annales de l'Institut Henri Poincare (B) Probability and Statistics},
  volume={36},
  pages={419--434},
  year={2000},
  organization={Elsevier}
}

@article{borell1974,
author = "Borell, C.",
doi = "10.1007/BF02384761",
journal = "Ark. Mat.",
number = "1-2",
pages = "239--252",
publisher = "Institut Mittag-Leffler",
title = "Convex measures on locally convex spaces",
url = "https://doi.org/10.1007/BF02384761",
volume = "12",
year = "1974"
}

@InProceedings{liKuelbs,
author="Li, W. V.
and Kuelbs, J.",
editor="Eberlein, Ernst
and Hahn, Marjorie
and Talagrand, Michel",
title="{Some Shift Inequalities for Gaussian Measures}",
booktitle="High Dimensional Probability",
year="1998",
publisher="Birkh{\"a}user Basel",
address="Basel",
pages="233--243"
}

@article{Oleszkiewicz,
author = {Rafał Latała and Krzysztof Oleszkiewicz},
title = {{Gaussian Measures of Dilatations of Convex Symmetric Sets}},
volume = {27},
journal = {Ann. Probab.},
number = {4},
publisher = {Institute of Mathematical Statistics},
pages = {1922 -- 1938},
year = {1999},
doi = {10.1214/aop/1022874821},
URL = {https://doi.org/10.1214/aop/1022874821}
}

@article{latala2002,
	author = {R. Latała},
	title = {{On some inequalities for Gaussian measures}},
	journal = {Proceedings of the ICM, Beijing},
	volume = {2},
	pages = {813--822},
	year = {2002}
}

@article{grohs2021stables,
  title={Stable Gabor phase retrieval for multivariate functions},
  author={Grohs, Philipp and Rathmair, Martin},
  journal={J. Eur. Math. Soc. (JEMS)},
  volume={24},
  number={5},
  pages={1593--1615},
  year={2021}
}

@article{chafaiMalrieu,
	author = {Djalil Chafaï and Florent Malrieu},
	title = {{On fine properties of mixtures with respect to concentration of measure and Sobolev type inequalities}},
	journal = {Ann. Inst. H. Poincaré Probab. Statist.},
	volume = {46},
	issue = {1},
	year = {2010},
	pages = {72--96},
	doi = {10.1214/08-AIHP309}
}

@article{schlichting,
	author = {Schlichting, A.},
	title = {{Poincaré and Log–Sobolev Inequalities for Mixtures}},
	journal = {Entropy},
	volume = {21},
	issue = {89},
	year = {2019},
	doi = {10.3390/e21010089},
	url = {http://doi.org/10.3390/e21010089}
}

@article{berry2018,
	author = {John Berry and Matthew Dannenberg and Jason Liang and Yingyi Zeng},
	title = {{The isoperimetric problem in the plane with the sum of two Gaussian densities}},
	journal = {Involve},
	volume = {11},
	number = {4},
	pages = {549--567},
	year = {2018},
	doi ={10.2140/involve.2018.11.549},
	url = {https://doi.org/10.2140/involve.2018.11.549}
}

@article{grohsRathmair,
	author = {P. Grohs and M. Rathmair},
	title = {{Stable Gabor Phase Retrieval and Spectral Clustering}},
	journal = {Comm. Pure Appl. Math.},
	volume = {72},
	number = {5},
	pages = {981--1043},
	year = {2019},
	doi ={10.1002/cpa.21799},
	url = {https://doi.org/10.1002/cpa.21799}
}

@article{alaifari2025cheeger,
  title={{Cheeger's constant for the Gabor transform and ripples}},
  author={Alaifari, Rima and Pineau, Ben and Taylor, Mitchell A and Wellershoff, Matthias},
  journal={arXiv:2512.18058},
  year={2025}
}

@article{fuhr2025cheeger,
  title={{Cheeger Bounds for Stable Phase Retrieval in Reproducing Kernel Hilbert Spaces}},
  author={F{\"u}hr, Hartmut and Getter, Max},
  journal={arXiv:2512.24169},
  year={2025}
}

@article{kuelbs1994gaussian,
  title={{The Gaussian measure of shifted balls}},
  author={Kuelbs, James and Li, Wenbo V and Linde, Werner},
  journal={Probab. Theory Related Fields},
  volume={98},
  number={2},
  pages={143--162},
  year={1994},
  publisher={Springer}
}

@article{chambers2019,
	author = {Gregory R. Chambers},
	title = {{Proof of the Log-Convex Density Conjecture}},
	journal = {J. Eur. Math. Soc. (JEMS)},
	volume = {21},
	pages = {2301--2332},
	year = {2019},
	doi ={10.4171/JEMS/885},
	url = {https://doi.org/10.4171/JEMS/885}
}

@article{cianchi2011,
author = {Cianchi, Andrea and Fusco, N and Maggi, F and Pratelli, And},
year = {2011},
pages = {131--186},
title = {{On the isoperimetric deficit in Gauss space}},
volume = {133},
number = {1},
journal = {Amer. J. Math.},
doi = {10.1353/ajm.2011.0005},
url = {https://doi.org/10.1353/ajm.2011.0005}
}

@article{fusco2011isoperimetric,
  title={{On the isoperimetric problem with respect to a mixed Euclidean-Gaussian density}},
  author={Fusco, Nicola and Maggi, Francesco and Pratelli, Aldo},
  journal={J. Funct. Anal.},
  volume={260},
  number={12},
  pages={3678--3717},
  year={2011},
  publisher={Elsevier}
}

@book{maggi2012sets,
  title={{Sets of finite perimeter and geometric variational problems: an introduction to Geometric Measure Theory}},
  author={Maggi, Francesco},
  volume={135},
  year={2012},
  publisher={Cambridge University Press}
}

@article{ambrosio2016minkowski,
  title={{Perimeter as relaxed Minkowski content in metric measure spaces}},
  author={Ambrosio, Luigi and Di Marino, Simone and Gigli, Nicola},
  journal={Nonlinear Anal.},
  volume={153},
  pages={78--88},
  year={2017},
  publisher={Elsevier}
}

@article{bobkov1996functional,
  title={{A functional form of the isoperimetric inequality for the Gaussian measure}},
  author={Bobkov, Sergey},
  journal={J. Funct. Anal.},
  volume={135},
  number={1},
  pages={39--49},
  year={1996},
  publisher={Elsevier}
}

@article{mosselNeeman2015,
author = {Elchanan Mossel and Joe Neeman},
title = {{Robust dimension free isoperimetry in Gaussian space}},
volume = {43},
journal = {Ann. Probab.},
number = {3},
publisher = {Institute of Mathematical Statistics},
pages = {971 -- 991},
year = {2015},
doi = {10.1214/13-AOP860},
URL = {https://doi.org/10.1214/13-AOP860}
}

@article{BROCK2016375,
title = {{An isoperimetric inequality for Gauss-like product measures}},
journal = {J. Math. Pures Appl.},
volume = {106},
number = {2},
pages = {375-391},
year = {2016},
issn = {0021-7824},
doi = {https://doi.org/10.1016/j.matpur.2016.02.014},
url = {https://www.sciencedirect.com/science/article/pii/S0021782416000155},
author = {F. Brock and F. Chiacchio and A. Mercaldo}
}

@article{bobkov1997some,
  title={{Some connections between Sobolev-type inequalities and isoperimetry}},
  author={Bobkov, Serguei G and Houdr{\'e}, Ch},
  journal={Mem. Amer. Math. Soc},
  volume={616},
  year={1997}
}

@article{liehrFollowUp,
  title={{On the uniqueness of Cheeger sets for Gaussian mixtures}},
  author={Lukas Liehr},
  journal={in preparation},
  year={2026}
}

@article{betta2008,
author = {Betta, M. Francesca and Brock, Friedemann and Mercaldo, Anna and Posteraro, M. Rosaria},
title = {{Weighted isoperimetric inequalities on {$\R^n$} and applications to rearrangements}},
journal = {Math. Nachr.},
volume = {281},
number = {4},
pages = {466-498},
doi = {10.1002/mana.200510619},
url = {https://onlinelibrary.wiley.com/doi/abs/10.1002/mana.200510619},
eprint = {https://onlinelibrary.wiley.com/doi/pdf/10.1002/mana.200510619},
year = {2008}
}

@article{canete2010,
	Author = {Ca{\~n}ete, Antonio and Miranda, Michele and Vittone, Davide},
	Doi = {10.1007/s12220-009-9109-4},
	Journal = {J. Geom. Anal.},
	Number = {2},
	Pages = {243--290},
	Title = {{Some Isoperimetric Problems in Planes with Density}},
	Url = {https://doi.org/10.1007/s12220-009-9109-4},
	Volume = {20},
	Year = {2010}
	}

@article{rosales2008,
	Author = {Rosales, C{\'e}sar and Ca{\~n}ete, Antonio and Bayle, Vincent and Morgan, Frank},
	Doi = {10.1007/s00526-007-0104-y},
	Journal = {Calc. Var. Partial Differential Equations},
	Number = {1},
	Pages = {27--46},
	Title = {{On the isoperimetric problem in Euclidean space with density}},
	Url = {https://doi.org/10.1007/s00526-007-0104-y},
	Volume = {31},
	Year = {2008}
	}

@article{marini,
  TITLE = {{On a class of weighted Gauss-type isoperimetric inequalities and applications to symmetrization}},
  AUTHOR = {Marini, Michele and Ruffini, Berardo},
  URL = {https://hal.archives-ouvertes.fr/hal-00941120},
  JOURNAL = {{Rend. Semin. Mat. Univ. Padova}},
  HAL_LOCAL_REFERENCE = {GT},
  PUBLISHER = {{University of Padua  / European Mathematical Society}},
  VOLUME = {133},
  PAGES = {197--214},
  YEAR = {2015},
  DOI = {10.4171/RSMUP/133-10},
  HAL_ID = {hal-00941120},
  HAL_VERSION = {v2}
}

@article{Betta1999_2,
author = {Betta, M.F. and Brock, F. and Mercaldo, A. and Possteraro, M.R.},
journal = {J. Inequal. Appl.},
language = {eng},
number = {3},
pages = {215-240},
publisher = {Springer International Publishing},
title = {A weighted isoperimetric inequality and applications to symmetrization.},
url = {http://eudml.org/doc/123150},
volume = {4},
year = {1999}
}

\end{document}